\documentclass{amsart}

\usepackage{amsmath, amssymb}

\newtheorem{theorem}{Theorem}[section]

\theoremstyle{definition}

\theoremstyle{remark}
\newtheorem{remark}[theorem]{Remark}

\numberwithin{equation}{section}

\newcommand{\N}{\mathbb{N}}                        
\newcommand{\C}{\mathbb{C}}                        
\newcommand{\K}{\mathbb{K}}                        
\newcommand{\uk}{\underline{k}}                    
\newcommand{\vp}{\varphi}                          
\newcommand{\vpn}{\vp^{\{n\}}}                     
\newcommand{\Xpn}{X^{\{n\}}}                       
\newcommand{\pp}{\mathfrak{p}}                     
\newcommand{\qq}{\mathfrak{q}}                     
\newcommand{\sss}{\mathfrak{s}}                    
\newcommand{\kp}{\kappa}                           
\newcommand{\tensR}{\otimes_R}                     
\newcommand{\mm}{\mathfrak{m}}                     
\newcommand{\nn}{\mathfrak{n}}                     
\newcommand{\ma}{\mathfrak{a}}                     
\newcommand{\Ass}{\mathrm{Ass}}                    
\newcommand{\Spec}{\mathrm{Spec}\,}                
\newcommand{\trdeg}{\mathrm{tr.deg}}               

\begin{document}

\title[Auslander criterion for openness of an algebraic morphism]{Geometric Auslander criterion for openness\\ of an algebraic morphism}

\author{Janusz Adamus, Edward Bierstone and Pierre D. Milman}
\address{J. Adamus, Department of Mathematics, The University of Western Ontario, London, Ontario, Canada N6A 5B7
         - and - Institute of Mathematics, Faculty of Mathematics and Computer Science, Jagiellonian University, ul. {\L}ojasiewicza 6, 30-348 Krak{\'o}w, Poland}
\email{jadamus@uwo.ca}
\address{E. Bierstone, The Fields Institute, 222 College Street, Toronto, Ontario,
Canada M5T 3J1 - and - Department of Mathematics, University of Toronto, Toronto, 
Ontario, Canada M5S 2E4}
\email{bierston@fields.utoronto.ca}
\address{P.D. Milman, Department of Mathematics, University of Toronto, Toronto, 
Ontario, Canada M5S 2E4}
\email{milman@math.toronto.edu}
\thanks{Research partially supported by Natural Sciences and Engineering 
Research Council of Canada Discovery Grant OGP 355418-2008 and Polish Ministry of 
Science Discovery Grant N N201 540538 (Adamus), and NSERC Discovery Grants
OGP 0009070 (Bierstone), OGP 0008949 (Milman)}


\begin{abstract}
We give a topological analogue for openness of a criterion for flatness 
that originates with Auslander.
Over a normal base of dimension $n$, failure of openness is detected by a \emph{vertical} component in the $n$'th fibred power of the morphism.
\end{abstract}

\maketitle

\section{Introduction}
\label{sec:intro}

The purpose of this note is to give a topological analogue for openness of a 
criterion for flatness that originates in a classical paper of 
Auslander~\cite{Au}.
Consider a morphism $\vp:X\to Y$ of schemes of finite type over a field, 
where $\vp$ is locally of finite type.
We show that, if $Y$ normal and of dimension $n$, then failure of openness of 
$\vp$ is equivalent to a severe discontinuity of the fibres --- the existence 
of an irreducible component of the source whose image is nowhere dense 
in $Y$ --- after passage to the $n$-fold fibred power of the mapping 
(Theorem \ref{thm:open-schemes}). The criterion is effective (see
Remark \ref{rem:effective}).

In the analogous criterion for (non-)flatness, one replaces the 
``existence of an irreducible component of the source'' with ``existence of 
an associated component of the source'', where the associated component can 
be isolated or embedded. Auslander's criterion is for flatness of finitely 
generated modules over a ring. The flatness criterion was recently extended 
to algebraic or analytic morphisms over the complex numbers, by the 
authors~\cite{ABM}, and then to algebraic morphisms over arbitrary fields, 
by Avramov and Iyengar~\cite{AI}. (See Section~\ref{sec:flatness}.) 

Let $Y$ be a normal scheme of finite type over a field $\uk$. (\emph{Normal} means that
every local ring of $Y$ is a normal ring; i.e., an integrally closed domain.)
Let $\vp: X\to Y$ be a morphism which is locally of finite type. 
We say that an associated component (isolated or embedded) of $X$ is \emph{vertical}
if its image lies in a proper subvariety of $Y$ (equivalently, is nowhere-dense in $Y$). 
If $Y=\Spec R$ and $X=\Spec A$, where $R$ is a normal algebra of finite
type over $\uk$ and $A$ is an $R$-algebra of finite type, and $\vp$ is the induced morphism,
then $X$ has a vertical associated component 
(resp. vertical irreducible component) over $Y$ if and only if
there is an associated prime (resp. minimal associated prime) 
$\pp$ in $\Spec A$ with $\pp\cap R\neq(0)$.
Equivalently, by the prime avoidance lemma \cite[Lemma\,3.3]{Eis}, 
$A$ (resp. $A_{\mathrm{red}}$) has a zero-divisor $r\in R$.
The following is our main result. (See \cite[Thm.\,2.2]{A1} for the 
complex analytic case.)

\begin{theorem}
\label{thm:open-schemes}
Let $Y$ be a scheme of finite type over a field $\uk$, 
and let $\vp:X\to Y$ denote a morphism which is locally of finite type. 
Assume that $Y$ is normal, of dimension $n$. Then $\vp$ is open
if and only if there are no vertical irreducible components in the $n$-fold 
fibred power $\vpn:\Xpn\to Y$.
\end{theorem}

Openness is a local property; i.e., a mapping $\vp:X\to Y$ is open
if and only if it is open in a neighbourhood of $\xi$, for every $\xi\in X$.
Theorem~\ref{thm:open-schemes} therefore reduces to the following affine criterion.

\begin{theorem}
\label{thm:open-main}
Let $R$ be an algebra of finite type over a field $\uk$, and let $A$ be an $R$-algebra of finite type.
Assume that $R$ is normal, of dimension $n$.
Then the induced morphism of spectra $\Spec A\to\Spec R$ is open 
if and only if the reduced $n$-fold tensor power
$\left(A^{\tensR^n}\right)_{\mathrm{red}}$ is a torsion-free $R$-module.
(Equivalently, $\pp\cap R=(0)$ for every minimal prime $\pp$ in $A^{\tensR^n}$.)
\end{theorem}

\begin{remark}
\label{rem:effective}
By Theorem~\ref{thm:open-main}, to verify non-openness of the
morphism $\Spec A\to\Spec R$,
it suffices to find a minimal associated prime in $A^{\tensR^n}$ which contains a nonzero element $r\in R$.
Similarly, by Theorem~\ref{thm:flat-main} (below), in order to verify that a finite $A$-module $F$ is not $R$-flat, it is enough to find an associated prime of 
$F^{\tensR^n}$ in $A^{\tensR^n}$ which contains a nonzero element $r\in R$. 
Thus Theorems~\ref{thm:open-main} and~\ref{thm:flat-main} together with Gr\"obner-basis
algorithms for primary decomposition (see, e.g., \cite{V2}) provide tools
for checking openness and flatness by effective computation.
\end{remark}

\section{Criterion for flatness}
\label{sec:flatness}

Theorem~\ref{thm:open-main} is a topological analogue of the criterion for flatness mentioned above.
In contrast to the general openness criterion, however, the flatness results are known for a smooth target only. On the other hand, flatness makes sense also for modules over a morphism, and the corresponding flatness criterion holds also in this context.

The flatness criterion (Theorem \ref{thm:flat-main} following) has its origin 
in Auslander~\cite{Au}. It was recently proved by the authors for complex 
analytic morphisms and, as a consequence, for morphisms of schemes of finite 
type over $\C$, using transcendental methods \cite{ABM} (see also \cite{GK}), 
and more recently for schemes of finite type over an arbitrary field, 
by Avramov and Iyengar \cite{AI} using homological methods (which require 
a smooth base ring).  The flatness criterion over any field of characteristic 
zero can, in fact, be deduced from the complex case, using the 
Tarski-Lefschetz Principle. In other words, the following is a corollary 
of \cite[Thm.\,1.3]{ABM}.

\begin{theorem}
\label{thm:flat-main}
Let $R$ be an algebra of finite type over a field $\uk$ of characteristic zero. 
Let $A$ denote an $R$-algebra essentially of finite type, and let $F$ denote a finitely generated $A$-module.
Assume that $R$ is regular, of dimension $n$.
Then $F$ is $R$-flat if and only if the $n$-fold tensor power $F^{\tensR^n}$ is a torsion-free $R$-module.
(Equivalently, $\pp\cap R=(0)$ for every associated prime $\pp$ of $F^{\tensR^n}$ in $A^{\tensR^n}$.)
\end{theorem}

An $R$-algebra \emph{essentially of finite type} means a localization of an $R$-algebra of finite type.
In the case that $\uk=\C$, the above result is \cite[Thm.\,1.3]{ABM}.
\medskip

The ``only if'' direction in Theorem~\ref{thm:flat-main} is immediate, since any tensor power
of a flat $R$-module is $R$-flat, and hence a torsion-free $R$-module,
by the characterization of flatness in terms of relations (\cite[Cor.\,6.5]{Eis}).
To deduce the ``if'' direction of Theorem~\ref{thm:flat-main} from~\cite[Thm.\,1.3]{ABM}, we proceed in two steps.
\par

First, assume that $\uk$ is algebraically closed.
Then our result follows from \cite[Thm.\,1.3]{ABM}, by the Tarski-Lefschetz Principle (see, e.g., \cite{Se}),
as flatness can be expressed in terms of a finite number of relations (\cite[Cor.\,6.5]{Eis}).
\par

Next, suppose that $\uk$ is an arbitrary field of characteristic zero,
and let $\K$ denote an algebraic closure of $\uk$. Let $A$ be an $R$-algebra essentially of finite type,
and let $F$ be a finitely generated $A$-module, which is not flat over $R$.
Put $R':=R\otimes_{\uk}\K$. Then $R'$ is a faithfully flat $R$-module.
Indeed, $R'$ is $R$-flat, because flatness is preserved by any base change and $\K$ is trivially
$\uk$-flat, as a $\uk$-vector space.
To prove faithful flatness, by \cite[Ch.\,I, \S\,3.5, Prop.\,9]{Bou},
it suffices to show that for every maximal ideal $\mm$ of $R$
there exists a maximal ideal $\nn$ of $R'$ such that $\nn\cap R=\mm$.
Let $\mm$ be an arbitrary maximal ideal in $R$. Then $\mm$ induces a proper ideal in $R'$. Let $\nn$ be a maximal ideal of $R'$
which contains the ideal induced by $\mm$ in $R'$. Since $R'$ is a homomorphic image of a polynomial
algebra $\K[y_1,\dots,y_n]$, with $\K$ algebraically closed, it follows that $\nn$ is
(the equivalence class of) an ideal of the form $(y_1-a_1,\dots,y_n-a_n)$ for some
$(a_1,\dots,a_n)\in\K^n$. Then clearly $\nn\cap R\subsetneq R$, and hence $\nn\cap R=\mm$, 
since
$\nn\cap R\supset\mm$ and $\mm$ is maximal.
\par
Put $A':=A\otimes_RR'$, and $F':=F\otimes_RR'$. Since $F$ is not $R$-flat, 
$F'$ is not $R'$-flat, by faithful flatness of $R'$ over $R$. Hence $F'^{\otimes^n_{R'}}$ has
a zero-divisor in $R'$, by the first part of the proof.
Therefore $F'^{\otimes^n_{R'}}$ has a non-zero associated prime $\qq$ in $R'$.
But $F'^{\otimes^n_{R'}} \cong (F^{\tensR^n})\otimes_RR'$, and for any $R$-module $M$ we have
\[
\Ass_{R'}(M\otimes_RR')\ =\ \bigcup_{\pp\in\Ass_R(M)}\Ass_{R'}(R'/\pp R')\,,
\]
by \cite[Ch.\,IV, \S\,2.6, Thm.2]{Bou}. Thus $F^{\tensR^n}$ has a non-zero associated prime in $R$;
i.e., $F^{\tensR^n}$ has a zero-divisor in $R$.
\qed

\section{Zero-divisors in tensor powers and variation of fibre dimension}
\label{sec:non-open}

Let $\vp:X\to Y$ denote a morphism of schemes of finite type over an arbitrary field $\uk$. 
We show that nontrivial variation of the dimension of the fibres of $\vp$ implies the
existence of a vertical component in a sufficiently high fibred 
power of $\vp$ (i.e., the existence of an irreducible component of a fibred 
power of $\vp$ which is mapped into a proper closed subscheme of $Y$).
We will use the following two classical results.
\smallskip

(1) Chevalley's theorem on upper semicontinuity of fibre dimension 
\cite[Thm.\,13.1.3]{Gro}: If $\vp:X\to Y$ is locally of finite type and 
$e \in \N$, then $F^{\vp}_e(X) := \{x\in X:\, \dim_x\vp^{-1}(\vp(x))\geq e\}$
is closed.
\smallskip

(2) The \emph{altitude formula} of Nagata: Suppose that $R$ is a Noetherian 
integral domain and that $A$ is an integral domain and an $R$-algebra of finite type. 
If $\pp$ is a prime ideal in $A$ and $\qq=\pp\cap R$, then
\begin{equation}
\label{eqn:Nagata}
\dim{A_{\pp}}\ \leq\ \dim{R_{\qq}}+\trdeg_{K(R)}K(A)-\trdeg_{\kp(\qq)}\kp(\pp)\,,
\end{equation}
with equality if $R$ is universally catenary. (See \cite[Thm.\,15.6]{M} or 
\cite[Thm.\,13.8]{Eis}.)

Here $\kp(\qq)$ and $\kp(\pp)$ denote the residue fields $R_{\qq}/\qq R_{\qq}$ 
and $A_{\pp}/\pp A_{\pp}$ (respectively), $K(R)$ and $K(A)$ are the fields of 
fractions of $R$ and $A$ (respectively), and $\trdeg_KL$ denotes the 
transcendence degree of a field extension $K\subset L$. A ring $R$ is called 
\emph{catenary} if, for every pair of prime ideals $\qq_1\subset\qq_2$ in $R$, 
all maximal chains of primes between $\qq_1$ and $\qq_2$ have the same length. 
$R$ is called \emph{universally catenary} if every finitely generated 
$R$-algebra is catenary.
\smallskip

Equality in (\ref{eqn:Nagata}) holds, for example, under the hypotheses of
Theorems~\ref{thm:open-schemes} and~\ref{thm:open-main} --- every algebra
of finite type over a field is universally catenary \cite[Cor.\,13.4]{Eis}.

Recall finally that if $Y$ is a normal scheme of finite type over a field and $X$
is irreducible, then a morphism $\vp:X\to Y$ which is locally of finite type is open
if and only if $\vp$ is dominating and the fibres of $\vp$ are equidimensional
and of constant dimension (by Grothendieck \cite[Prop.\,5.2.1 and Cor.\,14.4.6]{Gro}).

\begin{proof}[Proof of Theorem~\ref{thm:open-main}]
Let $R$ be an algebra of finite type over a field $\uk$, and let $A$ be a finite type $R$-algebra.
Assume that $R$ is normal, of dimension $n$. Let $X:=\Spec A$, $Y:=\Spec R$,
and let $\vp:X\to Y$ be the canonical morphism induced by the $R$-algebra structure of $A$.
\par

First suppose that $\vp$ is open. Then $\vp$ is \emph{universally open} (see 
\cite[D\'ef.\,14.3.3]{Gro}), by \cite[Cor.\,14.4.3]{Gro}. Thus
$X\times_YX\to X$ is open, so that $\vp^{\{2\}}:X\times_YX\to Y$ is open since
it is the composite of the open mappings $X\times_YX\to X$ and $X\to Y$.
By induction, all fibred powers $\vp^{\{k\}}:X^{\{k\}}\to Y$, $k\geq1$ are open. 
In particular, $\vpn|_W$ is dominating, for every isolated irreducible component $W$ of 
$\Xpn$; i.e., there are no vertical components.

Now suppose that $\vp$ is not open; hence not open at some point $\xi\in X$.
Let $\{Z_j\}$ denote the set of isolated irreducible components of $X$ at $\xi$.
We can assume that each $\vp|_{Z_j}$ is dominating. (Otherwise, already $X$
would have an isolated vertical component at $\xi$. Then, for each $k$, the fibred power $X^{\{k\}}$ would have an isolated vertical component at the corresponding diagonal point, since $X$ embeds into $X^{\{k\}}$ as the diagonal. More precisely, if $Z_j$ is an isolated vertical component of $X$, then $Z_j^{\{k\}}$ is an isolated vertical component of $X^{\{k\}}$.)

Since $\vp$ is not open at $\xi$, $\vp|_{Z_j}$ is not open at $\xi$, for each $j$.
Let $m:=\dim_{\xi}\vp^{-1}(\vp(\xi))$. By Chevalley's theorem,
$m$ is the maximal fibre dimension in some open neighbourhood $U$ of $\xi$.
We can assume that $U=X$ (since our problem is local). 

Let $F^{\vp}_m(X) := \{x\in X:\, \dim_x\vp^{-1}(\vp(x))=m\}$.
For each $j$, set
\[
F^{\vp}_m(Z_j):=\{x\in Z_j:\dim_x(\vp|_{Z_j})^{-1}(\vp(x))=m\}\,.
\]
Clearly, $F^{\vp}_m(X)=\bigcup_j F^{\vp}_m(Z_j)$. By Grothendieck's criterion
mentioned before the proof, each $F^{\vp}_m(Z_j)$ is a proper closed subset of $Z_j$ 
(perhaps empty). For each $j$, put
\[
d_j:=\dim{Z_j}\,, \quad m_j:=\min_{x\in Z_j}\dim_x(\vp|_{Z_j})^{-1}(\vp(x))\,.
\]
Then $\dim F^{\vp}_m(Z_j)\leq d_j-1$, since $Z_j$ is irreducible.

For each $j$, $d_j=\dim Y+m_j=n+m_j$, by equality in (\ref{eqn:Nagata}). 
If $F^{\vp}_m(Z_j) \neq \emptyset$, then $\vp|_{F^{\vp}_m(Z_j)}$ has
generic fibre dimension $m$ and it follows that
\[
\dim\vp(F^{\vp}_m(Z_j))\ \leq\ (d_j-1)-m\ =\ (n+m_j-1)-m\ \leq\ n-1\,,
\]
by equality in (\ref{eqn:Nagata}) again. Therefore,
$\vp(F^{\vp}_m(X))=\bigcup_j\vp(F^{\vp}_m(Z_j))$ is a constructible set (\cite[Ch.\,IV, Thm.\,1.8.4]{Gro}) of 
dimension at most $n-1$, so it lies in a proper closed subset $Y'$ of $Y$.

Let $\eta = \vp(\xi)$. Let $W$ be an irreducible component of the $n$-fold
fibred power $\Xpn$ which contains an $nm$-dimensional irreducible component 
of the fibre $(\vp^{\{n\}})^{-1}(\eta)$. Let $\sss$ be the (minimal) prime ideal in $A^{\tensR^n}$ 
which defines $W$ (i.e., $W=\Spec(A^{\tensR^n}/\sss)$). We claim that $\sss\cap R\neq(0)$
(i.e., $W$ is mapped by $\vpn$ into a proper subvariety of $Y$).

Suppose that $\sss\cap R=(0)$ (that is, suppose $W$ is dominant).
Then $A^{\tensR^n}/\sss\supset R$, and $A^{\tensR^n}/\sss\tensR K(R)$
is an algebra of finite type over $K(R)$. Hence
\[
\trdeg_{K(R)}K(A^{\tensR^n}/\sss) = \dim A^{\tensR^n}/\sss\tensR K(R)\,.
\]
The latter is the common length of all maximal chains of prime ideals in 
$A^{\tensR^n}/\sss\tensR K(R)$~\cite[\S\,8.2, Thm.\,A]{Eis}. On the other hand, 
the fibres of $\vpn|_W$ over the (nonempty) open set $Y\setminus Y'$ 
are of dimension at most $n(m-1)$; hence, the preceding common length of maximal 
chains of primes is at most $n(m-1)$. Therefore,
\begin{equation}
\label{eqn:trdeg}
\trdeg_{K(R)}K(A^{\tensR^n}/\sss)\leq n(m-1)\,.
\end{equation}
By (\ref{eqn:Nagata}) and (\ref{eqn:trdeg}), for every maximal ideal 
$\ma$ in $A^{\tensR^n}/\sss$, we have
\begin{multline}
\notag
\dim(A^{\tensR^n}/\sss)_{\ma}\ \leq\ \dim R_{\ma\cap R}+\trdeg_{K(R)}K(A^{\tensR^n}/\sss)\\
\leq\ \dim R\,+\,n(m-1)\ =\ nm\,.
\end{multline}
Hence $\dim W=\dim(A^{\tensR^n}/\sss)\leq nm$. Since $nm = \dim (\vpn)^{-1}(\eta)$ and
$W$ is irreducible, it follows that the generic points of $W$ and $(\vpn)^{-1}(\eta)$ coincide.
Therefore $\vpn(W)=\{\eta\}$; a contradiction.
\end{proof}

\bibliographystyle{amsplain}

\end{document}